\documentclass[11pt,english]{amsart}
\usepackage[T1]{fontenc}
\usepackage[latin1]{inputenc}
\usepackage{amstext}
\usepackage{amsthm}
\usepackage{amssymb}

\makeatletter
\numberwithin{equation}{section}
\numberwithin{figure}{section}
\theoremstyle{plain}
\newtheorem{thm}{\protect\theoremname}
  \theoremstyle{plain}
  \newtheorem{lem}[thm]{\protect\lemmaname}
  \theoremstyle{remark}
  \newtheorem{rem}[thm]{\protect\remarkname}
  \theoremstyle{plain}
  \newtheorem{prop}[thm]{\protect\propositionname}
  \theoremstyle{definition}
  \newtheorem{problem}[thm]{\protect\problemname}

\makeatother

\makeatother

\usepackage{babel}
  \providecommand{\lemmaname}{Lemma}
  \providecommand{\problemname}{Problem}
  \providecommand{\propositionname}{Proposition}
  \providecommand{\remarkname}{Remark}
\providecommand{\theoremname}{Theorem}

\begin{document}

\title[Curves with low Harbourne constants]{Curves with low Harbourne constants on Kummer and abelian surfaces}

\addtolength{\textwidth}{0mm}
\addtolength{\hoffset}{-0mm} 
\addtolength{\textheight}{0mm}
\addtolength{\voffset}{-0mm} 

%\subjclass{Primary: 14J29}% ; Secondary: 14G10, 14G15}

\global\long\global\long\def\Alb{{\rm Alb}}
 \global\long\global\long\def\Jac{{\rm Jac}}
\global\long\global\long\def\Disc{{\rm Disc}}

\global\long\global\long\def\Tr{{\rm Tr}}
 \global\long\global\long\def\NS{{\rm NS}}
\global\long\global\long\def\PicVar{{\rm PicVar}}
\global\long\global\long\def\Pic{{\rm Pic}}
\global\long\global\long\def\Br{{\rm Br}}
 \global\long\global\long\def\Pr{{\rm Pr}}
\global\long\global\long\def\Km{{\rm Km}}
\global\long\global\long\def\rk{{\rm rk}}

\global\long\global\long\def\Hom{{\rm Hom}}
 \global\long\global\long\def\End{{\rm End}}
 \global\long\global\long\def\aut{{\rm Aut}}
 \global\long\global\long\def\NS{{\rm NS}}
 \global\long\global\long\def\SSm{{\rm S}}
 \global\long\global\long\def\psl{{\rm PSL}}
 \global\long\global\long\def\CC{\mathbb{C}}
 \global\long\global\long\def\BB{\mathbb{B}}
 \global\long\global\long\def\PP{\mathbb{P}}
 \global\long\global\long\def\QQ{\mathbb{Q}}
 \global\long\global\long\def\RR{\mathbb{R}}
 \global\long\global\long\def\FF{\mathbb{F}}
 \global\long\global\long\def\DD{\mathbb{D}}
 \global\long\global\long\def\NN{\mathbb{N}}
 \global\long\global\long\def\ZZ{\mathbb{Z}}
 \global\long\global\long\def\HH{\mathbb{H}}
 \global\long\global\long\def\Gal{{\rm Gal}}
 \global\long\global\long\def\OO{\mathcal{O}}
 \global\long\global\long\def\pP{\mathfrak{p}}
 \global\long\global\long\def\pPP{\mathfrak{P}}
 \global\long\global\long\def\qQ{\mathfrak{q}}

\global\long\global\long\def\mm{\mathcal{M}}
 \global\long\global\long\def\aaa{\mathfrak{a}}
\global\long\def\a{\alpha}
\global\long\def\b{\beta}
 \global\long\def\d{\delta}
 \global\long\def\D{\Delta}
\global\long\def\L{\Lambda}
 \global\long\def\g{\gamma}
 \global\long\def\G{\Gamma}
 \global\long\def\d{\delta}
 \global\long\def\D{\Delta}
 \global\long\def\e{\varepsilon}
 \global\long\def\k{\kappa}
 \global\long\def\l{\lambda}
 \global\long\def\m{\mu}
 \global\long\def\o{\omega}
 \global\long\def\p{\pi}
 \global\long\def\P{\Pi}
 \global\long\def\s{\sigma}
 \global\long\def\S{\Sigma}
 \global\long\def\t{\theta}
 \global\long\def\T{\Theta}
 \global\long\def\f{\varphi}
 \global\long\def\deg{{\rm deg}}
 \global\long\def\det{{\rm det}}
 \global\long\def\dps{{\displaystyle }}
 \global\long\def\Dem{D\acute{e}monstration: }
 \global\long\def\ker{{\rm Ker\,}}
 \global\long\def\im{{\rm Im\,}}
 \global\long\def\rg{{\rm rg\,}}
 \global\long\def\car{{\rm car}}
\global\long\def\fix{{\rm Fix( }}
 \global\long\def\card{{\rm Card\  }}
 \global\long\def\codim{{\rm codim\,}}
 \global\long\def\coker{{\rm Coker\,}}
 \global\long\def\mod{{\rm mod }}
 \global\long\def\pgcd{{\rm pgcd}}
 \global\long\global\long\def\qa{\mathfrak{a}}
 \global\long\def\Sing{{\rm Sing}}

\global\long\def\ppcm{{\rm ppcm}}
 \global\long\def\la{\langle}
 \global\long\def\ra{\rangle}

\subjclass[2000]{Primary: 14J28 }

\author{Xavier Roulleau}
\begin{abstract}
We construct and study curves with low H-constants on abelian and
K3 surfaces. Using the Kummer $(16_{6})$-configurations on Jacobian
surfaces and some $(16_{10})$-configurations of curves on $(1,3)$-polarized
Abelian surfaces, we obtain examples of configurations of curves of
genus $>1$ on a generic Jacobian K3 surface with H-constants $<-4$. 
\end{abstract}

\maketitle

\section{Introduction}

The Bounded Negativity Conjecture predicts that for a smooth complex
projective surface $X$ there exists a bound $b_{X}$ such that for
any reduced curve $C$ on $X$ one has 
\[
C^{2}\geq b_{X}.
\]
That conjecture holds in some cases, for instance if $X$ is an abelian
surface, but we do not know whether it remains true if one considers
a blow-up of $X$. With that question in mind, the H-constants have
been introduced in \cite{BandCo}.\\
 For a reduced (but not necessarily irreducible) curve $C$ on a surface
$X$ and $\mathcal{P}\subset X$ a finite non empty set of points,
let $\pi:\bar{X}\to X$ be the blowing-up of $X$ at $\mathcal{P}$
and let $\bar{C}$ denotes the strict transform of $C$ on $\bar{X}$.
Let us define the number
\[
H(C,\mathcal{P})=\frac{\bar{C}^{2}}{|\mathcal{P}|},
\]
where $|\mathcal{P}|$ is the order of $\mathcal{P}$. We define the
Harbourne constant of $C$ (for short the H-constant) by the formula
\[
H(C)=\inf_{\mathcal{P}}H(C,\mathcal{P})\in\RR,
\]
where $\mathcal{P}\subset X$ varies among all finite non-empty subsets
of $X$ (note that there is a slight difference with the definition
of Hadean constant of a curve given in \cite[Remark 2.4]{BandCo},
which definition exists only for singular curves; see Remark \ref{rem:SLIGHT}
for the details). Singular curves tend to have low H-constants. It
is in general difficult to construct curves having low H-constants,
especially if one requires the curve to be irreducible. The (global)
Harbourne constant of the surface $X$ is defined by 
\[
H_{X}=\inf_{C}H(C)\in\RR\cup\{-\infty\}
\]
where the infimum is taken among all reduced curves $C\subset X$.
Harbourne constants and their variants are intensively studied (see
e.g. \cite{BandCo}, \cite{PT}, \cite{P}, \cite{RoulleauIMRN});
note that the finiteness of $H_{X}$ implies the BNC conjecture. Using
some elliptic curve configurations in the plane \cite{RU}, it is
known that
\[
H_{\PP^{2}}\leq-4,
\]
and for any surface $X$ one has $H_{X}\leq H_{\PP^{2}}\leq-4$ (see
\cite{RoulleauIMRN}). However, the curves $(C_{n})_{n\in\NN}$ on
$X\neq\PP^{2}$ with H-constant tending to $-4$ used to prove that
$H_{X}\leq-4$ are not very explicit and they all satisfy $H(C_{n})>-4$.\\
The H-constant is an invariant of the isogeny class of an abelian
surface. Using the classical $(16_{6})$ configuration $R_{1}$ of
$16$ genus $2$ curves and $16$ $2$-torsion points  in a principally
polarized abelian surface and a $(16_{10})$ configuration of $16$
smooth genus $4$ curves and $16$ $2$-torsion points on a $(1,3)$-polarized
abelian surface, plus the dynamic of the multiplication by $n\in\ZZ$
map, we construct explicitly some curves with low H-constants on abelian
surfaces:
\begin{thm}
\label{thm:MAIN}Let $A$ be a simple abelian surface. There exists
a sequence of curves $(R_{n})_{n\in\NN}$ in $A$ such that $R_{n}^{2}\to\infty$
and $H(R_{n})=-4$. \\
If $A$ is the Jacobian of a smooth genus $2$ curve, the curve $R_{n}$
can be chosen either as the union of $16$ smooth curves or as an
irreducible singular curve.
\end{thm}
It is known that on two particular abelian surfaces with CM there
exists a configuration $C$ of elliptic curves with $H(C)=-4$. Moreover
for any elliptic curve configuration $C$ in an abelian surface $A$,
one always has 
\[
H(C)\geq-4,
\]
with equality if and only if the complement of the singularities of
$C$ is an open ball quotient surface (for these previous results
see \cite{RoulleauIMRN}). Thus elliptic curve configurations with
$H(C)=-4$ are rather special, in particular these configurations
are rigid. Indeed to an algebraic family $(A_{t},C_{t})_{t}$ of such
surfaces $A_{t}$, each containing a configuration $C_{t}$ of elliptic
curves with $H$-constant equals to $-4$, such that $C_{t}$ varies
algebraically with $A_{t}$, one can associate a family of ball quotient
surfaces. Since ball quotient surfaces are rigid, the family $(A_{t},C_{t})_{t}$
is trivial and the pairs $(A_{t},C_{t})$ are isomorphic. \\
We observe that for our new examples of curves with $H(C)=-4$ there
is no such links with ball quotient surfaces. Indeed the pairs $(A,C)$
we give such that $H(C)=-4$ have deformations.

We then consider the images of the curves $R_{n}$ in the associated
Kummer surface $X$ and we obtain:
\begin{thm}
\label{thm:Let--be2}Let $X$ be a Jacobian Kummer surface. For any
$n>1$, there are configurations $C_{n}$ of curves of genus $>1$
such that $H(C_{n})=-4\frac{n^{4}}{n^{4}-1}<-4$.
\end{thm}
The H-constants of curves (and some related variants such as the $s$-tuple
Harbourne constants) on K3 surfaces have been previously studied,
by example in \cite{LaPo} and \cite{P}. In \cite{LaPo}, R. Laface
and P. Pokora study transversal arrangements $\mathcal{C}$ of rational
curves on K3 surfaces and they give examples of configurations $\mathcal{C}$
with a low Harbourne constant. In their examples, one has $H(\mathcal{C})\geq-3.777$,
with the exception of two examples on the Schur quartic and the Fermat
quartic surfaces, both reaching 
\[
H(\mathcal{C})=-8.
\]
In the last section, we then turn our attention to irreducible curves
with low H-constants in abelian and Kummer surfaces, which are more
difficult to obtain, some of which have been recently constructed
in \cite{RS}.

\subsection*{Acknowledgments}

The author thanks T. Szemberg for sharing  his observation that the
H-constant of the Kummer configuration is $-4$ and P. Pokora for
a very useful correspondence and his many criticisms. The author moreover
thanks the referee for pointing out an error in a previous version
of this paper and numerous comments which allowed to improve the exposition
of this note

\section{Smooth hyperelliptic curves in abelian surfaces and H-constants\label{sec:Smooth-hyperelliptic-curves}}

\subsection{Preliminaries, Notations}

By \cite{BO}, an abelian surface $A$ contains a smooth hyperelliptic
curve $C_{0}$ of genus $g$ if and only if it is a generic $(1,g-1)$-polarized
abelian surface and $g\in\{2,3,4,5\}$. \\
In this section, we study the configurations of curves obtained by
translation of these hyperelliptic curves $C_{0}$ (of genus $2,3,4$
or $5$) by $2$-torsion points and by taking pull-backs by endomorphisms
of $A$. In the present sub-section, we recall some facts on the computation
of the H-constants and some notations.\vspace{1mm}

Let $C_{1},\dots,C_{t}$ be smooth curves in a smooth surface $X$
such that the singularities of $C=\sum_{j}C_{j}$ are \textit{ordinary}
(i.e. resolved after one blow-up). Let $\Sing(C)$ be the singularity
set of $C$; we suppose that it is non-empty. Let $f:\bar{X}\to X$
be the blow-up of $X$ at $\Sing(C)$. For each $p$ in $\Sing(C)$,
let $m_{p}$ be the multiplicity of $C$ (we say that such a singularity
$p$ is a $m_{p}$\textit{-point}) and let $E_{p}$ be the exceptional
divisor in $\bar{X}$ above $p$. Let us recall the following notation:
\[
H(C,\mathcal{P})=\frac{\bar{C}^{2}}{|\mathcal{P}|},
\]
where $\bar{C}$ is the strict transform of a curve $C$ in the blowing-up
surface at $\mathcal{P}\neq\emptyset$. The following formula is well
known:
\begin{lem}
\label{lem:calcul de H}Let $s$ be the cardinal of $\Sing(C)$. One
has
\[
H(C,\Sing(C))=\frac{C^{2}-\sum_{p\in\mathcal{P}}m_{p}^{2}}{s}=\frac{\sum_{j=1}^{t}C_{j}^{2}-\sum_{p\in\mathcal{P}}m_{p}}{s}.
\]
\end{lem}
\begin{proof}
One can compute $\bar{C}^{2}$ in two ways, indeed
\[
\bar{C}=f^{*}C-\sum m_{p}E_{p},
\]
thus $\bar{C}^{2}=C^{2}-\sum_{p\in\mathcal{P}}m_{p}^{2}$ (that formula
is valid for any configurations). But $\bar{C}=\sum_{i=1}^{t}\bar{C_{i}}=\sum_{i=1}^{t}(f^{*}C_{i}-\sum_{p\in C_{i}}E_{p})$,
and since the singularities are ordinary, the curves $\bar{C}_{i}$
are disjoint, thus
\[
\bar{C}^{2}=\sum_{i=1}^{t}\bar{C_{i}}^{2}=\sum_{i=1}^{t}C_{i}^{2}-\sum_{p\in\Sing(C)}m_{p},
\]
where we just use the fact that $\sum_{i=1}^{t}\sum_{p\in C_{i}}1=\sum_{p\in\Sing(C)}m_{p}$.
\end{proof}
Recall that we define the H-constant of a curve $C$ by the formula
\[
H(C):=\inf_{\mathcal{P}}H(C,\mathcal{P})\in\RR,
\]
where $\mathcal{P}\subset X$ varies among all finite non-empty subsets
of $X$.
\begin{rem}
\label{rem:SLIGHT}a) If $C$ is smooth one has $H(C)=\min(-1,C^{2}-1)$.
\\
b) In \cite[Remark 2.4]{BandCo} the Hadean constant of a singular
curve $C$ on a surface $X$ is defined by the formula
\[
H_{ad}(C):=\min_{\mathcal{P}\subset\Sing(C),\,\mathcal{P}\neq\emptyset}H(C,\mathcal{P}).
\]
Let $C$ be an arrangement of $n\geq2$ smooth curves intersecting
transversally (with at least one intersection point). In \cite{LafPok}
is defined and studied the quantity $H(X,C):=H(C,\Sing(C))$. An advantage
of our definition of H-constant is that it is defined for any curves.
Moreover with our definition, it is immediate that the global H-constant
of the surface $X$ satisfies $H(X)=\inf H(C)$, where the infimum
is taken over reduced curves $C$ in $X$.
\end{rem}
Let $m\in\NN^{*}$ and let $C\hookrightarrow X$ be a singular curve
having singularities of multiplicity $m$ only (this will be the case
for most of the curves in this paper). Let $s$ be the order of $\Sing(C)$. 
\begin{lem}
One has $H(C,\Sing(C))=\frac{C^{2}}{s}-m^{2}$. \\
The H-constant of $C$ is 
\[
H(C)=\min(-1,\,C^{2}-m^{2},\,H(C,\Sing(C))).
\]
\end{lem}
\begin{proof}
For integers $0\leq a\leq s$, $b\geq0$, $c\geq0$ such that $a+b+c>0$,
let $\mathcal{P}_{a,b,c}$ be a set of $a$ $m$-points, $b$ smooth
points of $C$ and $c$ points in $X\setminus C$. Let 
\[
H_{a,b,c}=H(C,\mathcal{P}_{a,b,c})=\frac{C^{2}-am^{2}-b}{a+b+c}.
\]
The border cases are $H_{1,0,0}=C^{2}-m^{2}$, $H_{0,1,0}=C^{2}-1$
and $H_{0,0,1}=C^{2}$. If $a<\frac{c+C^{2}}{m^{2}-1}$ (case which
occurs when $c$ is large) the function $b\to H_{a,b,c}$ is decreasing
and converging to $-1$ when $b\to\infty$. \\
 If $a\geq\frac{c+C^{2}}{m^{2}-1}$, the function $b\to H_{a,b,c}$
is increasing, thus if $a\neq0$, one has
\[
\inf_{b\geq0}H_{a,b,c}=H_{a,0,c}=\frac{C^{2}-am^{2}}{a+c},
\]
(note that even if $a<\frac{C^{2}}{m^{2}-1}$, one still has $\frac{C^{2}-am^{2}}{a+c}\geq-1$).
If $C^{2}-am^{2}>0$, $H_{a,0,c}$ is a decreasing function of $c$,
with limit $0$, otherwise this is an increasing function and the
infimum is attained for $c=0$, which gives $\frac{C^{2}-am^{2}}{a}$
(if $a=0$, one gets $C^{2}$). Then taking the minimum over $a$,
one obtains the result. 
\end{proof}
\vspace{1mm}

Let us recall (see \cite{Dolgachev}) that for $a,b,n,m\in\NN^{*}$,
a $(a_{n},b_{m})$\textit{-configuration} is the data of two sets
$A,B$ of order $a$ and $b$, respectively, and a relation $R\subset A\times B$,
such that $\forall\a\in A,\,\#\{(\a,x)\in R\}=n$ and $\forall\b\in B,\,\#\{(y,\b)\in R\}=m$.
One has $an=bm=\#R$. If $a=b$ and $n=m$, it is called a $(a_{n})$\textit{-configuration}.
If for $\a\neq\a'$ in $A$ the cardinality $\l$ of $\{(\a,x)\in R\}\cap\{(\a',x)\in R\}$
does not depend on $\a\neq\a'$, this is called a $(a_{n},b_{m})$\textit{-design}
and $m(n-1)=\l(a-1)$; $\l$ is called the \textit{type of the design}. 

\subsection{Construction of configuration from Genus $2$ curves\label{subsec: Genus 2}}

Let $A$ be a principally polarized abelian surface such that the
principal polarization $C_{0}$ is a smooth genus $2$ curve. One
can choose an immersion such that $0\in A$ is a Weierstrass point
of $C_{0}$. The configuration of the $16$ translates 
\[
C_{t}=t+C_{0},\,t\in A[2]
\]
of $C_{0}$ by the $2$ torsion points of $A$ is the famous $(16_{6})$
Kummer configuration: there are $6$ curves through each point in
$A[2]$, and each curve contains $6$ points in $A[2]$ (since $C_{t}C_{t'}=2$
for $t\neq t'$ in $A[2]$, it is even a $(16_{6})$-design of type
$2$). 

Let now $n>0$ be an integer and let $[n]:A\to A$ be the multiplication
by $n$ map on $A$. For $t\in A[2]$, let us define $D_{t}=[n]^{*}C_{t}$,
in other words
\[
D_{t}=\{x\,|\,nx\in C_{t}\}=\{x\,|\,nx+t\in C_{0}\}.
\]
Since $[n]$ is étale, the curve $D_{t}$ is a smooth curve, thus
it is irreducible since its components are the pull back of an ample
divisor. By \cite[Proposition 2.3.5]{LB}, since $C_{t}$ is symmetric
(i.e. $[-1]^{*}C_{t}=C_{t}$), one has $D_{t}\sim n^{2}C_{t}$ (in
particular $D_{t}^{2}=2n^{4}$). The curve 
\[
W_{n}=[n]^{*}\sum_{t\in A[2]}C_{t}=\sum_{t\in A[2]}D_{t}
\]
has $16$ irreducible components and $16n^{4}$ ordinary singularities
of multiplicity $6$ ($6$-points), which are the torsion points $A[2n]:=\ker[2n]$.
Each curve $D_{t}$ contains $6n^{4}$ $6$-points; the configuration
of curves $D_{t}$ and singular points of $W_{n}$ is a $(16_{6n^{4}},16n^{4}\,_{6})$-configuration.
Using Lemma \ref{lem:calcul de H}, we get:
\begin{lem}
One has $H(W_{n},\Sing(W_{n}))=-4.$
\end{lem}
The Harbourne constant $H_{A}$ of a surface $A$ is an invariant
of the isogeny class of $A$ (see \cite{RoulleauIMRN}). Thus if $A$
is generic, it is isogeneous to the Jacobian of a smooth genus $2$
curve, and we thus obtain the following:
\begin{prop}
On a generic abelian surface $A$, one has: 
\[
H_{A}\leq-4.
\]
\end{prop}
Note that when $A$ is isogeneous to the product of $2$ elliptic
curves $E,E'$ (thus non generic in our situation), the H-constant
of $A$ verifies that $H_{A}\leq-2$, and $H_{A}\leq-3$ if $E$ and
$E'$ are isogeneous (see \cite{RoulleauIMRN}). Moreover, there are
two examples of surfaces with CM for which $H_{A}\leq-4$.
\begin{rem}
i) Suppose that $n$ is odd, then 
\[
D_{t}=\{x\,|\,n(x+t)\in C_{0}\}=D_{0}+t.
\]
Moreover, if $u$ is a $2$-torsion point one has $2u=0\Leftrightarrow2nu=0$,
thus $D_{0}$ and each curve $D_{t}$ contains $6$ points of $2$-torsion.
\\
ii) Suppose that $n$ is even. Let $u\in A[2]$ be a $2$-torsion
point. One has $u\in D_{t}\Leftrightarrow nu+t\in C_{0}\Leftrightarrow t\in C_{0}$.
Therefore the $6$ curves $D_{t}$ with $t$ in $A[2]\cap C_{0}$
contain $A[2]$, and the remaining curves do not contain any points
from $A[2]$.
\end{rem}

\subsection{Genus $3$ curves}

Let $A$ be an abelian surface containing a hyperelliptic genus $3$
curve $C_{0}$ such that $0$ is a Weierstrass point. Then the $8$
Weierstrass points of $C_{0}$ are contained in the set of $2$-torsion
points of $A$. Let $\OO$ be the orbit of $C_{0}$ under the action
of $A[2]$ by translation and let $a$ be the cardinal of $\OO$.
The stabilizer $S_{t}$ of $C_{0}$ acts as a fix-point free automorphism
group of $C_{0}$. Thus considering the possibilities for the genus
of $C_{0}/S_{t}$ it is either trivial or an involution, therefore
$a=16$ or $8$. By \cite[Remark 1]{BS}, the curve $C_{0}$ is stable
by translation by a $2$-torsion point, therefore $a=8$. Let $m$
be the number of curves in $\OO$ through one point in $A[2]$ (this
is well defined because $A[2]$ acts transitively on itself). The
sets of $8$ genus $3$ curves and $A[2]$ form a $(8_{8},16_{m})$-configuration,
thus $m=4$. Moreover, since they are translates, two curves $C,C'\in\OO$
satisfy $CC'=C^{2}=4$, thus 
\[
C\sum_{C'\in\OO,\,C'\neq C}C'=7\cdot4.
\]
If the singularities of the union of the curves in $\OO$ were only
at the points in $A[2]$ and ordinary, one would have 
\[
C\sum_{C'\in\OO,\,C'\neq C}C'=8\cdot3.
\]
 The configuration $\mathcal{C}=\sum_{C\in\OO}C$ contains therefore
other singularities than the points in $A[2]$ or the singularities
are non ordinary. It seems less interesting from the point of view
of H-constants. Observe that if the singularities at $A[2]$ are ordinary,
one has $H(\mathcal{C},A[2])=-2$. If there are other singularities,
since the configuration is stable by translations by $A[2]$, there
are at least $16$ more singularities. 

\subsection{Construction of configurations from Genus $4$ curves\label{subsec: Genus 4}}

Traynard in \cite{Traynard}, almost one century later Barth, Nieto
in \cite{BN}, and Naruki in \cite{Naruki} constructed $(16_{10})$
configurations of lines lying on a $3$-dimensional family of quartic
K3 surfaces $X$ in $\PP^{3}$: there exist two sets $\mathcal{C},\,\mathcal{C}'$
of $16$ disjoint lines in $X$ such that each line in $\mathcal{C}$
meets exactly $10$ ten lines in $\mathcal{C}'$, and vice versa.
\\
By the famous results of Nikulin characterizing Kummer surfaces, there
exists a double cover $\pi:\tilde{A}\to X$ branched over $\mathcal{C}$.
That cover contains $16$ $(-1)$-curve over $\pi^{-1}\mathcal{C}$.
The contraction $\mu:\tilde{A}\to A$ of these $16$ exceptional divisors
is an abelian surface and the image of these $16$ curves is the set
$A[2]$ of two torsion points of $A$. \\
We denote by $C_{1},\dots,C_{16}$ the $16$ smooth curves images
by $\mu_{*}\pi^{*}$ of the $16$ disjoint lines in $\mathcal{C}'$.
By \cite[Section 6]{BN}, the $16$ curves $C_{1},\dots,C_{16}$ are
translates of each other by the action by the group $A[2]$ of $2$-torsion
points; the argument is that if $C'_{i}$ is a translate of $C_{i}$
by a $2$-torsion point, then $\pi_{*}\mu^{*}C_{i}'$ is a line in
the quartic $X$, but a such a generic quartic has exactly $32$ lines.
\begin{prop}
\label{prop:There-exists-a 16_10}The curves $C_{1},\dots,C_{16}$
in $A$ are smooth of genus $4$. The $16$ $2$-torsion points $A[2]$
and these $16$ curves form a $(16_{10})$-design of type $6$: $10$
curves though one point in $A[2]$, a curve contains $10$ points
in $A[2]$ and two curves meet at $6$ points in $A[2]$. \\
The H-constant of that configuration $\sum C_{i}$ is $H=-4.$
\end{prop}
\begin{proof}
The $10$ intersection points between the lines in $\mathcal{C}$
and $\mathcal{C}'$ are transverse, therefore by the Riemann-Hurwitz
Theorem, the genus of the $16$ irreducible components of $\pi^{*}\mathcal{C}'$
is $4$. The intersections of the $16$ components in $\mu_{*}\pi^{*}\mathcal{C}'$
are transverse (since $\pi^{*}\mathcal{C}'$ is a union of disjoint
curves) and that intersection holds over points in $A[2]$ (which
is the image of the exceptional divisors of $\tilde{A}$). \\
Since the curves in $\mathcal{C}$ and $\mathcal{C}'$ form a $(16_{10})$
configuration, the $16$ curves $C_{1},\dots,C_{16}$ and the $2$
torsion points in $A$ have the described $(16_{10})$ configuration.
\\
Since the strict transform in $\tilde{A}$ of the curves $C_{i}\neq C_{j}$
are two disjoint curves, the $6$ intersection points of $C_{i}\neq C_{j}$
are $2$-torsion points, the configuration is therefore a $(16_{10})$-design
of type $6$. \\
It is then immediate to compute the H-constant of $\mathcal{C}=C_{1}+\dots+C_{16}$.
\end{proof}
\begin{rem}
\label{rem:change of notations} Since the $16$ curves are the orbit
of a curve by the group $A[2]$ of torsion points, one can change
the notations and define $C_{t}=C_{0}+t$ for $t\in A[2]$, for a
chosen curve $C_{0}$ containing $0$. As in sub-section \ref{subsec: Genus 2},
let us define $D_{t}=[n]^{*}C_{t}$; this is a smooth curve. It is
then immediate to check that the $H$-constant of the curve $W_{n}=\sum D_{t}$
equals $-4$. We will use these configurations of curves in Section
\ref{sec:Configurations-of-curves 3}.
\end{rem}

\subsection{Genus $5$ curves}

By \cite{BO}, a generic $(1,4)$-polarized abelian surface contains
a smooth genus $5$ curve $C$ which is hyperelliptic, the set of
Weierstrass points in $C$ is $12$ $2$-torsion points, and $C$
is stable by a sub-group of $A[2]$ isomorphic to $(\ZZ/2\ZZ)^{2}$.
Thus the orbit of $C$ by the translations by elements of $A[2]$
is the union of $4$ genus $5$ curves. \\
The intersection of two of these curves equals $C^{2}=2g-2=8$. Since
each of these two curves contains $12$ points in $A[2]$, the intersections
are transverse and are on $8$ points in $A[2]$. The $4$ curves
and the $16$ $2$-torsion points form a $(4_{12},16_{3})$ configuration.
The $H$-constant of that configuration is $H=\frac{4\cdot8-16\cdot3}{16}=-1.$

\section{Configurations of curves with low H-constant in Kummer surfaces\label{sec:Configurations-of-curves 3}}

In this Section, we study the images in the Kummer surface $\Km(A)$
of the various curve configurations studied in Section \ref{sec:Smooth-hyperelliptic-curves}
in abelian surfaces $A$.

\subsection{ The genus $2$ case}

We keep the notations and hypothesis of sub-section \ref{subsec: Genus 2}.
In particular, $A$ is the Jacobian of a genus $2$ curve. Let $\mu:\tilde{A}\to A$
be the blow-up of $A$ at the $16$ $2$-torsion points. We denote
by $\bar{D}$ the strict transform in $\tilde{A}$ of a curve $D\hookrightarrow A$.
Let $\pi:\tilde{A}\to X$ be the quotient map by the automorphism
$[-1]$. Since on $A$ one has $[-1]^{*}C_{t}=[-1]^{*}(t+C_{0})=C_{t}$,
one obtains 
\[
[-1]^{*}D_{t}=D_{t}
\]
and the map $\bar{D_{t}}\to D'_{t}=\pi(\bar{D}_{t})$ has degree $2$,
thus $D_{t}'^{2}=\frac{1}{2}(\bar{D_{t}})^{2}$. 
\begin{prop}
Let be $n>1$. The configuration $\mathcal{D}$ of the $16$ curves
$D'_{t}$ with $t\in A[2]$ in the Kummer surface $X$ has Harbourne
constant 
\[
H(\sum_{t\in A[2]}D'_{t})=-4\frac{n^{4}}{n^{4}-1}.
\]
\end{prop}
\begin{proof}
If $n$ is even, a curve $D_{t}$ contains $16$ or $0$ points of
$2$ torsion depending if $t\in C_{0}$ or not (thus there are $10$
curves without points of $2$ torsion, and $6$ with). If $n$ is
odd, each curve $D_{t}$ contains $6$ points of $2$-torsion and
then one has: 
\[
D_{t}'^{2}=\frac{1}{2}(2n^{4}-6)=n^{4}-3.
\]
If $n$ is even, one has: 
\[
D_{t}'^{2}=\frac{1}{2}(2n^{4}-16)=n^{4}-8\,\mbox{or }D'{}_{t}^{2}=n^{4},
\]
according if $t\in A[2]$ is in $C_{0}$ or not. The  configuration
$\mathcal{D}$ contains 
\[
\frac{1}{2}(16n^{4}-16)=8(n^{4}-1)
\]
$6$-points and no other singularities. If $n$ is even, then the
configuration has $10$ curves with self-intersection $n^{4}$ and
$6$ curves with self-intersection $n^{4}-8$. Thus if $n$ is even
one has
\[
H(\mathcal{D})=\frac{10n^{4}+6(n^{4}-8)-8(n^{4}-1)6}{8(n^{4}-1)}=-4\frac{n^{4}}{n^{4}-1}\sim-4,
\]
which for $n=2$ gives $H=-64/15\simeq-4.2\bar{6}$. 

If $n$ is odd, one has $16$ curves with self-intersection $n^{4}-3$,
and we get the same formula: 
\[
H(\mathcal{D})=\frac{\sum D_{t}'^{2}-8(n^{4}-1)6}{8(n^{4}-1)}=\frac{16(n^{4}-3)-8(n^{4}-1)6}{8(n^{4}-1)}=-4\frac{n^{4}}{n^{4}-1}.
\]
\end{proof}
\begin{rem}
a) The H-constants of the various configurations are $<-4$.\\
b) For $n=1$, the H-constant is $-2$.
\end{rem}

\subsection{The genus $4$ case}

Let us consider the configuration $(16_{10})$ considered in sub-section
\ref{subsec: Genus 4} of $16$ genus $4$ curves $C_{t},\,t\in A[2]$
in a generic $(1,3)$-polarized abelian surface $A$. Let $X=\Km(A)$
be the Kummer surface associated to $A$. Let $\mu:\tilde{A}\to A$
the blow-up at the points in $A[2]$, and $\pi:\tilde{\ensuremath{A}}\to X$
be the quotient map. Let us consider as in Remark \ref{rem:change of notations}
the $16$ curves $D_{t}=[n]^{*}C_{t}$, $t\in A[2]$ in $A$. Let
be $\bar{D}_{t}$ the strict transform in $\tilde{A}$ of $D_{t}$
and $D_{t}'=\pi(\bar{D}_{t})$. 
\begin{prop}
For $n>1$, the configuration $\mathcal{C}=\sum_{t\in A[2]}D'_{t}$
in the Kummer surface $X$ has Harbourne constant 
\[
H(\mathcal{C})=-4\frac{n^{4}}{n^{4}-1}.
\]
\end{prop}
\begin{proof}
The involution $[-1]:A\to A$ fixes the set $A[2]$ and stabilizes
the configuration $\mathcal{C}=\sum_{t\in A[2]}C_{t}$, since a curve
$C_{t}$ in $\mathcal{C}$ is determined by the $2$-torsion points
it contains, $[-1]$ stabilizes each curve $C_{t}$, $t\in A[2]$,
and thus also one has $[-1]^{*}D_{t}=D_{t}$. Thus the restriction
$\bar{D}_{t}\to D_{t}'$ of $\pi$ has degree $2$. Numerically, one
has $D_{t}=n^{2}C_{t}$ and $C_{t}^{2}=6$. \\
Since $D_{t}$ as a set is $\{x\in A\,|\,nx+t\in C_{0}\}$, a point
$t'\in A[2]$ is in $D_{t}$ if and only if $nt'+t\in C_{0}$. Thus
if $n$ is even, the curve $D_{t}$ contains $16$ or $0$ points
of $2$ torsion depending if $t\in C_{0}$ or not (thus there are
$6$ curves without points of $2$ torsion, and $10$ with), moreover
one has: 
\[
D_{t}'^{2}=\frac{1}{2}(\bar{D}_{t})^{2}=\frac{1}{2}(6n^{4}-16)=3n^{4}-8\,\,\mbox{or }\,D'{}_{t}^{2}=3n^{4},
\]
according if $t\in A[2]$ is in $C_{0}$ or not. If $n$ is odd, each
curve $D_{t}$ contains $10$ points of $2$-torsion and 
\[
D_{t}'^{2}=\frac{1}{2}(6n^{4}-10)=3n^{4}-5.
\]
 The configuration $\mathcal{C}=\sum D'_{t}$ has $\frac{1}{2}(16n^{4}-16)=8(n^{4}-1)$
$10$-points and no other singularities. If $n$ is even, then the
configuration contains $6$ curves with self-intersection $3n^{4}$
and $10$ curves with self-intersection $3n^{4}-8$, thus
\[
H(\mathcal{C})=\frac{6\cdot3n^{4}+10(3\cdot n^{4}-8)-8(n^{4}-1)10}{8(n^{4}-1)}=-4\frac{n^{4}}{n^{4}-1}\sim-4.
\]
If $n$ is odd, one has $16$ curves with self-intersection $3n^{4}-5$,
and one gets the same formula.
\end{proof}
\begin{rem}
The multiplication by $n$ map $[n]$ on $A$ induces a rational map
$[\frak{n}]:X\dashrightarrow X$. The configurations $\sum D_{t}$
we are describing are the pull back by $[\frak{n}]$ of a configuration
in $X=\Km(A)$ of $16$ disjoint rational curves. 
\end{rem}

\section{Irreducible curves with low H-constant in abelian and Kummer surfaces.}

Obtaining irreducible curves with low Harbourne constant is in general
a difficult problem. Let $k>0$ be an integer. In \cite{RS}, we prove
that in a generic abelian surface polarized by $M$ with $M^{2}=k(k+1)$
there exists a hyperelliptic curve $\Gamma_{k}$ numerically equivalent
to $4M$ such that $\Gamma_{k}$ has a unique singularity of multiplicity
$4k+2$. Thus:
\begin{prop}
The $H$-constant of $\Gamma_{k}$ is 
\[
H(\Gamma_{k})=\Gamma_{k}^{2}-(4k+2)^{2}=-4.
\]
\end{prop}
Let us study the case $k=1$ and define $T_{1}=\Gamma_{1}$. This
is a curve of geometric genus $2$ in an abelian surface $A$ with
one $6$-point singularity, which we can suppose in $0$. In \cite{PRR}
such a curve $T_{1}$ is constructed: $A$ is the Jacobian of a genus
$2$ curve $C_{0}$ and $T_{1}$ is the image by the multiplication
by $2$ map of $C_{0}$ in $A=J(C_{0})$. The self-intersection of
$T_{1}$ is $T_{1}^{2}=32$ and the singularity of $T_{1}$ is a $6$-point.
Let $n\in\NN$ be odd. The curve $T_{n}=[n]^{*}T_{1}\sim n^{2}T_{1}$
has $6$-points singularities at each points of $A[n]$, the set of
$n$-torsion points of $A$. Let $[n]$ be the multiplication by $n$
map. The following diagram of curve configurations in $A$ is commutative
\[
\begin{array}{ccc}
\sum_{t\in A[2]}D_{t} & \stackrel{[2]}{\to} & T_{n}\\
\downarrow[n] &  & \downarrow[n]\\
\sum_{t\in A[2]}C_{t} & \stackrel{[2]}{\to} & T_{1}
\end{array}
\]
 Since the multiplication by $2$ map $[2]$ has degree $16$ and
the curves $D_{t}$ are permuted by translations by elements of $A[2]$,
the map $D_{t}\stackrel{[2]}{\to}T_{n}$ is birational, thus $T_{n}$
is irreducible. Its singularities are $6$-points over each $n$-torsion
points. 
\begin{thm}
Let $n\in\NN$ be odd. The Harbourne constant of the irreducible curve
$T_{n}\hookrightarrow A$ is :
\[
H(T_{n})=\frac{32n^{4}-36n^{4}}{n^{4}}=-4.
\]
\end{thm}
Let $\bar{T_{n}}$ be the strict transform of $T_{n}$ under the blowing-up
map $\tilde{A}\to A$ at the points from $A[2]$. Since $n$ is odd,
among points in $A[2]$, $T_{n}$ contains only $0$, thus $\bar{T_{n}}^{2}=32n^{4}-36$.
Moreover $[-1]^{*}T_{n}=T_{n}$ (it can be seen using the map $C^{(2)}\to A$
that $[-1]^{*}T_{1}=T_{1}$, and therefore $[-1]^{*}T_{n}=T_{n}$).
The image of $\bar{T_{n}}$ on the Kummer surface $X=\tilde{A}/[-1]$
is an irreducible curve $W_{n}$ with $\text{\ensuremath{\frac{1}{2}}(}n^{4}-1)$
$6$-points if $n$ is odd. The map 
\[
\bar{T}_{n}\to W_{n}
\]
has degree $2$ and 
\[
W_{n}^{2}=16n^{4}-18,
\]
thus  
\begin{prop}
Let $n\in\NN$ be odd. The H-constant of the irreducible curve $W_{n}$
in the Kummer surface $X$ is
\[
H(W_{n})=\frac{-4n^{4}}{n^{4}-1}.
\]
In particular, for $n=3$ one has $H(W_{3})=-\frac{81}{20}$.
\end{prop}

\section{Some remarks on H-constants of abelian surfaces}

Let $\phi:C\hookrightarrow A$ be an irreducible curve of geometric
genus $g$ in an abelian surface $A$. Let $m_{p}=m_{p}(C)$ be the
multiplicity of $C$ at a point $p$. One has
\[
C^{2}=2g-2+2\g,
\]
where 
\[
\g\geq\sum_{p}\frac{1}{2}m_{p}(m_{p}-1),
\]
with equality if all singularities are ordinary. Thus 
\[
H(C,\Sing(C))=\frac{1}{\#\Sing(C)}(C^{2}-\sum_{p\in\Sing(C)}m_{p}^{2})\geq\frac{1}{\#\Sing(C)}(2g-2-\sum_{p\in\Sing(C)}m_{p})
\]
with equality if all singularities are ordinary. From the previous
construction in Section \ref{sec:Smooth-hyperelliptic-curves}, one
can ask the following
\begin{problem}
Does there exists an abelian surface containing a curve $C$ of geometric
genus $g$ such that 
\[
-4>\frac{1}{\#\Sing(C)}(2g-2-\sum_{p\in\Sing(C)}m_{p})\,\,?
\]
\end{problem}
We were not able to find any example of such a curve. If the answer
is no, it would imply that the bounded negativity conjecture holds
true for surfaces which are blow-ups of abelian surfaces.

\vspace{5mm}
\noindent Xavier Roulleau,
\\Aix-Marseille Universit\'e, CNRS, Centrale Marseille,
\\I2M UMR 7373,  
\\13453 Marseille, France
\\ {\tt Xavier.Roulleau@univ-amu.fr}
%\\ 
\vspace{0.5cm} 
\noindent \urladdr{https://old.i2m.univ-amu.fr/~roulleau.x/}
\end{document}